\documentclass[11pt, a4paper]{amsart}

\usepackage{amsfonts,amsmath,amssymb, amscd,fullpage}
\usepackage{stmaryrd}
\usepackage[all]{xy}

\newtheorem{theorem}{Theorem}[section]
\newtheorem{lemma}[theorem]{Lemma}
\newtheorem{definition}[theorem]{Definition}
\newtheorem{proposition}[theorem]{Proposition}
\newtheorem{corollary}[theorem]{Corollary}
\newtheorem{example}[theorem]{Example}
\newtheorem{remark}[theorem]{Remark}

\theoremstyle{definition}

\newcommand\pf{\begin{proof}}
\newcommand\epf{\end{proof}}

\newcommand\Irr{\mathrm{Irr}}

\numberwithin{equation}{section}

\title{Half-commutative orthogonal Hopf algebras}

\author{Julien Bichon}
\address{J. Bichon: 
Laboratoire de Math\'ematiques,
Universit\'e Blaise Pascal, Campus des c\'ezeaux BP 80026,
63171~Aubi\`ere Cedex, France}
\email{Julien.Bichon@math.univ-bpclermont.fr}
\author{Michel Dubois-Violette}
\address{M. Dubois-Violette:
Laboratoire de Physique Th\'eorique,
B\^atiment 210,
Universit\'e Paris XI, 
91405 Orsay Cedex, France}
\email{Michel.Dubois-Violette@u-psud.fr}


\begin{document}

\maketitle

\begin{abstract}
A half-commutative orthogonal Hopf algebra is a Hopf $*$-algebra generated by the self-adjoint coefficients
of an orthogonal matrix corepresentation $v=(v_{ij})$ that half commute in the sense that $abc=cba$ 
for any $a,b,c \in \{v_{ij}\}$. The first non-trivial such Hopf algebras were discovered by Banica and Speicher. We propose a general procedure, based on a crossed product construction, that associates to a  self-transpose compact subgroup $G \subset U_n$ 
a half-commutative  orthogonal Hopf algebra $\mathcal A_*(G)$. It is shown 
that any half-commutative orthogonal Hopf algebra arises in this way. 
The fusion rules of $\mathcal A_*(G)$
are expressed in term of those of $G$.
\end{abstract}

\section{introduction}

The half-liberated orthogonal quantum group $O_n^*$ were recently discovered by Banica and Speicher
\cite{bsp}. These are compact quantum groups in the sense of Woronowicz \cite{wo1}, and the
corresponding Hopf $*$-algebra $A_o^*(n)$ is the universal $*$-algebra presented by self-adjoint generators $v_{ij}$ submitted to the relations making $v=(v_{ij})$ an orthogonal matrix  and to the half-commutation relations
$$abc = cba, \  a,b,c \in \{v_{ij}\}$$
The half-commutation relations arose, via Tannaka duality, from a deep study of certain tensor subcategories of the category of partitions, see \cite{bsp}. More examples of Hopf algebras
with generators satisfying the half-commutation relations were given in \cite{bcs}.

The representation theory of $O_n^*$ was discussed in \cite{bve}, where strong links with the representation theory of the unitary group $U_n$ were found. It followed that the fusion rules of $O_n^*$ are non-commutative if $n\geq 3$. Moreover a matrix model $A_o^*(n) \hookrightarrow M_2(\mathcal R(U_n))$ was found in \cite{bcs2}.

The aim of this paper is to continue these works by a general study of what we call  half-commutative orthogonal Hopf algebras:  Hopf $*$-algebras generated by the self-adjoint coefficients
of an orthogonal matrix corepresentation $v=(v_{ij})$ whose coefficients satisfy the previous half-commutation relations.
Our main results are as follows.
\begin{enumerate}
 \item To any self-transpose compact subgroup $G \subset U_n$ 
we associate a half-commutative orthogonal Hopf algebra $\mathcal A_*(G)$, with 
$\mathcal A_*(U_n) \simeq A_o^*(n)$. The Hopf algebra $\mathcal A_*(G)$ is a Hopf $*$-subalgebra
of the crossed product $\mathcal R(G) \rtimes \mathbb C \mathbb Z_2$, where the action of $\mathbb Z_2$ of $\mathcal R(G)$ is induced by the transposition. 
\item Conversely we show that any noncommutative half-commutative orthogonal Hopf algebra arises from the previous construction for some compact group $G \subset U_n$.
\item We show that the fusion rules of $\mathcal A_*(G)$ can be described in terms of those of $G$.
\end{enumerate}

 Therefore it follows from our study that quantum groups arising from half-commutative orthogonal Hopf algebras are objects that are very close from classical groups. This was suggested by 
 the representation theory results from \cite{bve}, by the matrix model found in the ``easy'' case in \cite{bcs2} and by the results of \cite{bbcc} where it was shown
 that the quantum group inclusion $O_n \subset O_n^*$ is maximal.
 The techniques from \cite{bbcc}, and especially the short five lemma for cosemisimple Hopf algebras, are used in essential way here. The use of versions of the five lemma for  Hopf algebras was initiated in   \cite{aga}.
 

 The paper is organized as follows. In Section 2 we fix some notation and recall the necessary background.
 In Section 3 we formally introduce half-commutative orthogonal Hopf algebras, and recall the early examples from \cite{bsp,bcs}. Section 4 is devoted to our main construction, which associates
 to a self-transpose compact subgroup $G \subset U_n$ a half-commutative orthogonal Hopf algebra $\mathcal A_*(G)$, and we show that any half-commutative orthogonal Hopf algebra arises in this way.
 At the end of the section we use our construction to propose a possible orthogonal half-liberation of the unitary group $U_n$.
 In Section 5 we describe the fusion rules of $\mathcal A_*(G)$ in terms of those of $G$.
 
We assume that the reader is familiar with Hopf algebras \cite{mo}, Hopf $*$-algebras and with the algebraic approach (via algebras of representative functions) to compact quantum groups \cite{diko,ks}.

\section{preliminaries}

\subsection{Classical groups.} We first fix some notation. As usual, the group
of complex $n \times n$ unitary matrices is denoted by $U_n$, while $O_n$ denotes the group of 
real orthogonal matrices. We denote by $\mathbb T$ the subgroup of $U_n$ consisting of scalar matrices, and by $PU_n$ the quotient group $U_n/\mathbb T$. 

We shall need the following notions.

\begin{definition}\label{defclassical}
 Let $G \subset U_n$ be a compact subgroup.
 \begin{enumerate}
  \item We say that $G$ is self-transpose if $\forall g \in G$, we have $g^t \in G$.
  \item We say that $G$ is non-real if $G \not \subset O_n$, i.e. if there exists $g \in G$ with $g_{ij} \not \in \mathbb R$, for some $i,j$.
  \item We say that $G$ is doubly non-real if  there exists $g \in G$ with $g_{ij}\overline{g_{kl}} \not \in \mathbb R$, for some $i,j,k,l$.
 \end{enumerate} 
\end{definition}

Note that the subgroup $\tilde{O_n} = \mathbb T  O_n\subset U_n$ (considered in \cite{bbcc}) is non-real but is not doubly non-real.

\subsection{Orthogonal and unitary Hopf algebras.} In this subsection we recall some definitions on the algebraic approach to compact quantum groups. We work at the level of Hopf $*$-algebras of representative
functions. The following simple key definition arose from Woronowicz' work \cite{wo1}.

\begin{definition}
A \textit{unitary Hopf algebra} is a $*$-algebra $A$ which is generated by elements $\{u_{ij}|1 \leq i,j \leq n\}$ such that the matrices $u = (u_{ij})$ and $\overline{u}=(u_{ij}^*)$ are unitaries, and such that:
\begin{enumerate}
 \item There is a $*$-algebra map $\Delta:A \to A \otimes A$ such that $\Delta(u_{ij}) = \sum_{k=1}^n u_{ik} \otimes u_{kj}$.

\item There is a $*$-algebra map $\varepsilon:A \to \mathbb C$ such that $\varepsilon(u_{ij}) = \delta_{ij}$.

\item There is a $*$-algebra map $S:A \to A^{op}$ such that $S(u_{ij}) = u_{ji}^*$.
\end{enumerate}
If $u_{ij} = u_{ij}^*$ for $1 \leq i,j \leq n$, we say that $A$ is an \textit{orthogonal Hopf algebra}.
\end{definition}

It follows that $\Delta,\varepsilon,S$ satisfy the usual Hopf $*$-algebra axioms and that $u=(u_{ij})$ is a matrix corepresentation of $A$. Note that the definition forces that a unitary Hopf algebra is of Kac type, i.e. $S^2={\rm id}$. The motivating examples of unitary (resp. orthogonal) Hopf algebra is $A = \mathcal R(G)$, the algebra of representative functions on a compact subgroup $G \subset U_n$ (resp. $G \subset O_n$).  Here the standard generators $u_{ij}$ are the coordinate functions which take a matrix to its $(i,j)$-entry.

In fact every commutative unitary Hopf algebra is of the form $\mathcal R  (G)$ for some unique compact group $G \subset U_n$ defined by $G = {\rm Hom}_{*-alg}(A, \mathbb C)$ (this the Hopf algebra version of the Tannaka-Krein theorem).
This motivates the  notation ``$A = \mathcal R(G)$'' for any unitary (resp. orthogonal) Hopf algebra, where $G$ is a \textit{unitary (resp. orthogonal) compact quantum group}. 

The universal examples of unitary and orthogonal Hopf algebras are as follows \cite{wa1}.

\begin{definition}
The universal unitary Hopf algebra $A_u(n)$ is the universal $*$-algebra generated by  elements $\{u_{ij}|1 \leq i,j \leq n\}$ such that the matrices
$u = (u_{ij})$ and $\overline{u}=(u_{ij}^*)$ in $M_n(A_u(n))$ are unitaries.

The universal orthogonal Hopf algebra $A_o(n)$ is the universal $*$-algebra generated by self-adjoint elements $\{u_{ij}|1 \leq i,j \leq n\}$ such that the matrix $u = (u_{ij})_{1 \leq i,j \leq n}$ in $M_n(A_o(n))$ is orthogonal.
\end{definition}

The existence of the Hopf $*$-algebra structural morphisms follows from the universal properties of $A_u(n)$ and $A_o(n)$.  As discussed above, we use the notations $A_u(n)= \mathcal R(U_n^+)$ and 
$A_o(n) = \mathcal R(O_n^+)$, where $U_n^+$ is the \textit{free unitary quantum group} and $O_n^+$ is the \textit{free orthogonal quantum group}.

The Hopf $*$-algebra $A_u(n)$ was introduced by Wang \cite{wa1}, while 
the Hopf algebra $A_o(n)$ was defined first in \cite{dvl} under the notation
$\mathcal A(I_n)$, and was then defined independently in \cite{wa1} in the compact quantum group framework.

\subsection{Exact sequences of Hopf algebras}
In this subsection we recall some facts on exact sequences of Hopf algebras. 

\begin{definition}\label{prex}
A sequence of Hopf algebra maps
$$\mathbb C \to B\overset{i}\to A\overset{p}\to L\to\mathbb C$$
is called pre-exact if $i$ is injective, $p$ is surjective and $i(B)=A^{cop}$, where:
$$A^{cop}=\{a\in A|(id\otimes p)\Delta(a)=a\otimes 1\}$$
\end{definition}

A pre-exact sequence as in Definition \ref{prex} is said to be exact \cite{ade} if in addition we have $i(B)^+A=\ker(p)=Ai(B)^+$, where $i(B)^+=i(B)\cap\ker(\varepsilon)$. For the kind of sequences 
to be considered in this paper, pre-exactness is actually equivalent to exactness.

The following lemma, that we record for future use, is Proposition 3.2 in \cite{bbcc}.

\begin{lemma}\label{ex1}
Let $A$ be an orthogonal Hopf algebra with  generators $u_{ij}$. Assume that we have surjective Hopf algebra map $p:A\to\mathbb C\mathbb Z_2$, $u_{ij}\to\delta_{ij}g$, where $<g>=\mathbb Z_2$. Let $P_uA$ be the subalgebra generated by the elements $u_{ij}u_{kl}$ with the inclusion $i:P_uA \subset A$. Then the sequence
$$\mathbb C\to P_uA\overset{i}\to A\overset{p}\to\mathbb C\mathbb Z_2 \to
\mathbb C$$
is pre-exact.
\end{lemma}

Exact sequences of compact groups induce exact sequences of Hopf algebras. In particular
if $G \subset U_n$ is a compact subgroup, we have an exact sequence of compact groups
$$1\to G \cap \mathbb T \to G \to G /G \cap \mathbb T  \to 1$$
that induces an exact sequence of Hopf algebras
$$\mathbb C\to \mathcal R(G/G \cap \mathbb T) \to \mathcal R(G) \to \mathcal R(G \cap \mathbb T)  \to
\mathbb C$$
We will use the following probably well-known lemma. We sketch a proof for the sake of completeness.

\begin{lemma}\label{pun}
Let $G \subset U_n$ be a compact subgroup. Then $\mathcal R(G/G \cap \mathbb T)$
is the subalgebra of $\mathcal R(G)$ generated by the elements $u_{ij}u_{kl}^*$, $i,j,k, l \in \{1, \ldots , n\}$. Moreover, if $G=U_n$, then $\mathcal R(PU_n) = \mathcal R(U_n/ \mathbb T)$ is isomorphic with the commutative $*$-algebra presented by generators $w_{ij,kl}$, $1 \leq i,j,k,l \leq n$ and submitted to the relations
$$\sum_{j=1}^n w_{ik,jj}= \delta_{ik} = \sum_{j=1}^n w_{jj,ik}, \ w_{ij,kl}^*=w_{ji,lk}$$
$$\sum_{k,l=1}^n w_{ij,kl}w_{pq,kl}^* = \delta_{ip}\delta_{jq}$$
The isomorphism is given by $w_{ij,kl} \longmapsto u_{ik}u_{jl}^*$.
\end{lemma}

\begin{proof}
Let $p : \mathcal R(G) \longrightarrow \mathcal R(G\cap \mathbb T)$
 be the restriction map. It is clear ${\rm Ker}(p)$ is generated as a $*$-ideal by the elements $u_{ij}$, $i \not =j$,
 and $u_{ii}-u_{jj}$. 
 Let $B$ be the subalgebra generated by the elements $u_{ij}u_{kl}^*$.  Then $B$ is a Hopf $*$-subalgebra of $\mathcal R(G)$ and it is clear that
 $B \subset \mathcal R(G)^{cop}$. To prove the reverse inclusion
 we form the Hopf algebra quotient 
 $\mathcal R(G)// B = \mathcal R(G)/ B^+ \mathcal R(G)$ and denote by $\rho : \mathcal R(G) \longrightarrow \mathcal R(G)// B$ the canonical projection. It is not difficult to see
 that in $\mathcal R(G)// B$ we have $\rho(u_{ij})=0$ if $i\not=j$ and $\rho(u_{ii})=\rho(u_{jj})$ for any $i,j$. Hence there exists a Hopf $*$-algebra map $p' : \mathcal R(G /\mathbb T) \longrightarrow \mathcal R(G)// B$ such that $p' \circ p = \rho$. It follows that $\mathcal R(G)^{\rm cop} \subset \mathcal R(G)^{{\rm co}\rho}$. But since our algebras are commutative, $\mathcal R(G)$ is a faithfully flat $B$-module and hence by \cite{ta} (see also \cite{ade}) we have $\mathcal R(G)^{{\rm co}\rho} =B$, and hence  $\mathcal R(G/G \cap \mathbb T)=\mathcal R(G)^{\rm cop}=B$.
 
 The last assertion is just the reformulation of the standard fact that $PU_n$ is the automorphism group of the $*$-algebra $M_n(\mathbb C)$ (see e.g. \cite{wa3}).
\end{proof}



\section{Half-commutative Hopf algebras}

We now formally introduce half-commutative orthogonal Hopf algebras.
Of course the definition of half-commutativity can be given in a general context, as follows.
It was first formalized, in a probabilistic context, in \cite{bcs3}.

\begin{definition}
 Let $A$ be an algebra. We say that a family $(a_i)_{i\in I}$ of elements of $A$ half-commute
 if $abc=cba$ for any $a,b,c \in \{a_i, i \in  I\}$. The algebra $A$ is said to be half-commutative
 if it has a family of generators that half-commute. 
\end{definition}

At a Hopf algebra level, a reasonnable definition seems to be the following one.

\begin{definition}
 A half-commutative Hopf algebra is a Hopf algebra $A$ generated by the  coefficients
 of a matrix corepresentation $v=(v_{ij})$ whose coefficients half-commute.
 \end{definition}
 
 We will not study half-commutative Hopf algebras in this generality. 
 A reason for this is that it is unclear if the half-commutativity relations outside of the orthogonal case are the natural ones in the categorical framework of \cite{bsp}. Thus we will restrict to the following special
 case. 
 
 \begin{definition}
 A half-commutative orthogonal Hopf algebra is a Hopf $*$-algebra $A$ generated by the self-adjoint coefficients
 of an orthogonal matrix corepresentation $v=(v_{ij})$ whose coefficients half-commute.
\end{definition}

The first example is the universal one, defined in \cite{bsp}.

\begin{definition}
The half-liberated orthogonal Hopf algebra $A_o^*(n)$ is the universal $*$-algebra generated by self-adjoint elements $\{v_{ij}|1 \leq i,j \leq n\}$ which half-commute  and such that the matrix $v = (v_{ij})_{1 \leq i,j \leq n}$ in $M_n(A_o^*(n))$ is orthogonal.
\end{definition}

The existence of the Hopf algebra structural morphisms  follows from the universal property of $A_o^*(n)$, and hence $A_o^*(n)$ is a half-commutative orthogonal Hopf algebra.  We use the notation $A_o^*(n) = \mathcal R(O_n^*)$, where $O_n^*$ is the \textit{half-liberated orthogonal quantum group}.
Note that we have $\mathcal R(O_n^+) \twoheadrightarrow \mathcal R(O_n^*)  \twoheadrightarrow \mathcal R(O_n)$, i.e. $O_n \subset O_n^* \subset O_n^+$. At $n=2$ we have $O_2^*=O_2^+$, but for $n \geq 3$ these inclusions are strict.

Another example of half-commutative orthogonal Hopf algebra is the following one, taken  from \cite{bcs}.

\begin{definition}
 The half-liberated hyperoctaedral Hopf algebra $A_h^*(n)$ is the universal $*$-algebra generated by self-adjoint elements $\{v_{ij}|1 \leq i,j \leq n\}$ which half-commute, such that
 $v_{ij}v_{ik}=0=v_{ki}v_{ji}$ for $k\not=j$,   and such that the matrix $v = (v_{ij})_{1 \leq i,j \leq n}$ in $M_n(A_o^*(n))$ is orthogonal.
\end{definition}

Again the existence of the Hopf algebra structural morphisms follows from the universal property of $A_h^*(n)$, and hence $A_h^*(n)$ is a half-commutative orthogonal Hopf algebra. See \cite{bcs}
and \cite{we} for further examples.

The following lemma will be an important ingredient in the proof of the structure theorem of half-commutative orthogonal Hopf algebras.

\begin{lemma}\label{exacthalf}
Let $A$ be a half-commutative  orthogonal Hopf algebra
 generated by the self-adjoint coefficients
 of an orthogonal matrix corepresentation $v=(v_{ij})$ whose coefficients half-commute.
 Then $P_vA$ is a commutative Hopf $*$-subalgebra of $A$. If moreover
 $A$ is noncommutative then 
 there exists a Hopf $*$-algebra map $p : A \longrightarrow \mathbb C\mathbb Z_2$
 such that for any $i,j$, $p(v_{ij})= \delta_{ij}s$, where $\langle s \rangle= \mathbb Z_2$,  that induces a pre-exact sequence
  $$\mathbb C \to   P_vA \overset{i}\to  A \overset{p}\to \mathbb C\mathbb Z_2\to\mathbb C$$
\end{lemma}

\begin{proof}
 The key observation that $P_vA$ is commutative is Proposition 3.2 in \cite{bve}.
 It is clear that $P_vA$ is a normal Hopf $*$-subalgebra of $A$, and hence we can form the Hopf $*$-algebra quotient $A//P_vA= A/ A(P_vA)^+$, with $p : A \longrightarrow A//P_vA$ the canonical
 surjection. It is not difficult to see that in $A//P_vA$ we have $p(v_{ij})=0$ if $i \not =j$, $p(v_{ii})=p(v_{jj})$ for any $i,j$ and if we put $g=p(v_{ii})$, $g^2=1$. So we have to prove that
 $g \not =1$. If $g=1$, then $A//P_vA$ is trivial and $p= \varepsilon$.
 We know from \cite{chi} that $A$ is faithfully flat as a $P_vA$-module (since orthogonal Hopf algebras are cosemisimple), and hence by \cite{sch}, we have $A^{\rm cop} = P_vA$. So if $g=1$ we have $A^{\rm cop}=P_vA=A$ and $A$ is commutative. Thus if $A$ is noncommutative we have $g\not =1$, the map
 $p$ satisfies the conditions in the statement  and we have the announced exact sequence (Lemma \ref{ex1}).
\end{proof}

\begin{remark}
 {\rm The previous exact sequence is cocentral. Thus it is possible, in principle, to classify
 the finite-dimensional half-commutative orthogonal Hopf algebras according to the scheme used in \cite{bn}. The classification data will involve in particular pairs $(\Gamma, \omega)$ formed
by a finite subgroup $\Gamma \subset PU_n$ and 
 a cocycle $\omega \in H^2(\Gamma, \mathbb Z_2)$, see \cite{bn} for details. 
 }
\end{remark}

\section{The main construction}

In this section we perform our main construction that associates to any self-transpose compact subgroup $G \subset U_n$ 
a half-commutative orthogonal Hopf algebra $\mathcal A_*(G)$ and we show any half-commutative orthogonal Hopf algebra arises in this way.

We begin with a well-known lemma. We give a proof for the sake of completeness.

\begin{lemma}\label{s} 
 Let $G \subset U_n$ be a  compact subgroup, and denote by $u_{ij}$ the coordinate functions
 on $G$. 
 The following assertions are equivalent.
 \begin{enumerate}
 \item $G$ is self-transpose.
\item  There exists a unique involutive Hopf $*$-algebra automorphism 
$s : \mathcal R(G) \longrightarrow \mathcal R(G)$ such that $s(u_{ij}) =u_{ij}^*$. 
\end{enumerate}
Moreover if $G$ is self-transpose the automorphism is non-trivial if and only $G$ is non-real.
\end{lemma}

\begin{proof} Assume that $G$ is self-transpose. Then we have an involutive compact group automorphism
 \begin{align*}
 \sigma : G &\longrightarrow G \\
  g &\longmapsto (g^t)^{-1} = \overline{g}
 \end{align*}
 which induces an involutive Hopf $*$-algebra automorphism $s : \mathcal R(G) \longrightarrow \mathcal R(G)$ such that $s(u_{ij}) =u_{ij}^*$. Uniqueness is obvious since the elements $u_{ij}$ generate $\mathcal R(G)$ as a $*$-algebra. Conversely, the existence of $s$ will ensure the existence of the automorphism $\sigma$ since $G \simeq {\rm Hom}_{*-alg}(\mathcal R(G),\mathbb C)$, and hence $G$ will be self-transpose. The last assertion is immediate.
\end{proof}

\begin{definition}
 Let $G \subset U_n$ be a self-transpose non-real compact subgroup. 
 We denote by $\mathcal R(G) \rtimes \mathbb C\mathbb Z_2$ the crossed product Hopf $*$-algebra associated to the involutive Hopf $*$-algebra automorphism $s$ of Lemma \ref{s}.
\end{definition}

Recall that the Hopf $*$-algebra structure of $\mathcal R(G) \rtimes \mathbb C\mathbb Z_2$
is defined as follows (see e.g. \cite{ks}).

\begin{enumerate}
 \item As a coalgebra, $\mathcal R(G) \rtimes \mathbb C\mathbb Z_2 =\mathcal R(G) \otimes \mathbb C\mathbb Z_2$.
 \item We have $(f \otimes s^i) \cdot (g \otimes s^j) = fs^i(g) \otimes s^{i+j}$, for any $f,g \in \mathcal R(G)$ and $i,j \in \{0,1\}$.
 \item We have $(f \otimes s^i)^*= s^i(f)^* \otimes s^i$   for any $f \in \mathcal R(G)$ and $i \in \{0,1\}$.
 \item The antipode is given by $S(u_{ij} \otimes 1) = u_{ji}^* \otimes 1$, $S(u_{ij} \otimes s) = u_{ji} \otimes s$ (in short $S(f \otimes s^i) = s^i(S(f)) \otimes s^i$ for any $f \in \mathcal R(G)$ and $i \in \{0,1\}$). 
\end{enumerate}

For notational simplicity we denote, for $f \in \mathcal R(G)$,  the respective elements $f \otimes 1$ and $f \otimes s$ of $\mathcal R(G) \rtimes \mathbb C\mathbb Z_2$  by $f$ and $fs$.

\begin{definition}
 Let $G \subset U_n$ be a self-transpose  compact subgroup. 
 We denote by $\mathcal A_*(G)$ the subalgebra of $\mathcal R(G) \rtimes \mathbb C\mathbb Z_2$
 generated by the elements $u_{ij}s$, $i,j \in\{1, \ldots , n\}$.
\end{definition}

\begin{proposition}\label{rep}
  Let $G \subset U_n$ be a self-transpose compact subgroup. Then $\mathcal A_*(G)$ is a Hopf $*$-subalgebra of $\mathcal R(G) \rtimes \mathbb C\mathbb Z_2$, and there exists a surjective Hopf $*$-algebra
  morphism
  \begin{align*}
   \pi : A_o^*(n) &\longrightarrow \mathcal A_*(G) \\
   v_{ij} &\longmapsto u_{ij}s
  \end{align*}
 Hence $\mathcal A_*(G)$ is a half-commutative orthogonal Hopf algebra, and is noncommutative if and only if $G$ is doubly non-real.
\end{proposition}

\begin{proof}
 We have $(u_{ij}s)^* = su_{ij}^*= u_{ij}s$ and hence the elements $u_{ij}s$ are self-adjoint and generate a $*$-subalgebra. Moreover, using the coproduct and antipode formula, it is immediate to check that
  $\Delta(u_{ij}s) = \sum_ku_{ik}s \otimes u_{kj}s$ and $S(u_{ij}s)=u_{ji}s$, and hence 
 $\mathcal A_*(G)$ is an orthogonal  Hopf $*$-subalgebra of $\mathcal R(G) \rtimes \mathbb C\mathbb Z_2$.
 We have
 $$u_{ij}su_{kl}su_{pq}s = u_{ij}u_{kl}^*u_{pq}s=u_{pq}u_{kl}^* u_{ij}s=u_{pq}su_{kl}su_{ij}s$$
 Hence the coefficients of the orthogonal matrix $(u_{ij}s)$ half-commute, and we get our
 Hopf $*$-algebra map  $\pi : A_o^*(n) \longrightarrow \mathcal A_*(G)$. The algebra 
 $\mathcal A_*(G)$ is commutative if and only if the elements $u_{ij}s$ pairwise commute. We have
 $u_{ij}su_{kl}s = u_{ij}u_{kl}^*$, so $\mathcal A_*(G)$ is noncommutative if and only if there exist $i,j,k,l$
 with $u_{ij}u_{kl}^*\not =u_{kl}u_{ij}^*$, which precisely means that $G$ is doubly non-real.
\end{proof}

The Hopf $*$-algebra $\mathcal A_*(G)$ is part of a natural pre-exact sequence.


\begin{proposition}\label{exact} Let $G \subset U_n$ be a self-transpose compact subgroup.
 Then there exists a Hopf $*$-algebra embedding 
 $\mathcal R(G/G \cap \mathbb T) \hookrightarrow \mathcal A_*(G)$ and a pre-exact sequence
 $$\mathbb C \to  \mathcal R(G/G \cap \mathbb T) \overset{j}\to \mathcal A_*(G) \overset{q}\to \mathbb C\mathbb Z_2\to\mathbb C$$
\end{proposition}

\begin{proof}
 The map $q$ is defined as the restriction to $\mathcal A_*(G)$ of the Hopf $*$-algebra map
 $\varepsilon \otimes {\rm id} : \mathcal R(G) \rtimes \mathbb C\mathbb Z_2 \rightarrow \mathbb C\mathbb Z_2$.
 Hence we have $q(u_{ij}s) = \delta_{ij}s$. Let $B$ be the subalgebra of $\mathcal A_*(G)$ generated by the elements $u_{ij}su_{kl}s= u_{ij}u_{kl}^*$. It is clear that $B = \mathcal A_*(G)^{coq}$, and hence we have a pre-exact sequence
 $$\mathbb C \to  B \overset{j}\to \mathcal A_*(G) \overset{q}\to \mathbb C\mathbb Z_2\to\mathbb C$$
Consider now the injective Hopf algebra map $\nu : \mathcal R(G) \hookrightarrow \mathcal R(G) \rtimes \mathbb C\mathbb Z_2$, $f \mapsto f \otimes 1$. Since $\mathcal R(G/G \cap\mathbb T)= \mathcal R(G)^{G \cap \mathbb T}$
is the subalgebra generated by the elements $u_{ij}u_{kl}^*$ (Lemma \ref{pun}), we have $\nu(\mathcal R(G/\mathbb T))= B$,
and we get our pre-exact sequence. 
\end{proof}

We will prove  (Theorem \ref{struc}) that  a noncommutative half-commutative
orthogonal Hopf algebra is isomorphic to $\mathcal A_*(G)$ for some compact group $G \subset U_n$.
Before this we first prove that the morphism in Proposition \ref{rep} is an isomorphism
$A_o^*(n) \simeq \mathcal A_*(U_n)$. This can be seen as a consequence of the forthcoming
Theorem \ref{struc}, but the proof is less technical while it already well enlights the main ideas.

\begin{theorem}
 We have a Hopf $*$-algebra isomorphism  $A_o^*(n) \simeq \mathcal A_*(U_n)$. 
\end{theorem}

\begin{proof} Let $\pi : A_o^*(n) \longrightarrow \mathcal A_*(U_n)$ be the Hopf $*$-algebra map from Proposition \ref{rep}, defined by $\pi(v_{ij})=u_{ij}s$. It induces a commutative diagram of Hopf algebra maps with pre-exact rows
$$\begin{CD}
\mathbb C@>>>P_vA_o^*(n)@>{i}>>A_o^*(n)@>{p}>>\mathbb C\mathbb Z_2 @>>>\mathbb C\\
@.@VV{\pi_|}V@VV{\pi}V@|@.\\
\mathbb C@>>>\mathcal R(PU_n)@>{j}>>\mathcal A_*(U_n)@>{q}>>\mathbb C\mathbb Z_2@>>>\mathbb C\\
\end{CD}$$
where the sequence on the top row is the one of Lemma \ref{exacthalf} and  the sequence on the lower row is the one of Proposition \ref{exact}.
The standard presentation of $\mathcal R(PU_n)$ (Lemma \ref{pun}) ensures the existence of a $*$-algebra map
$\mathcal R(PU_n) \longrightarrow P_vA_o^*(n)$, $u_{ij}u_{kl}^* \mapsto v_{ij}v_{kl}$ which is clearly an inverse isomorphism for $\pi_|$. Thus we can invoke the short five lemma from \cite{bbcc} (Theorem 3.4) to conclude that $\pi$ is an isomorphism.
\end{proof}

Note that a precursor for the previous isomorphism $A_o^*(n) \simeq \mathcal A_*(U_n)$
was the matrix model  $A_o^*(n)\hookrightarrow M_2(\mathcal R(U_n))$ found in \cite{bcs2}, Section 8.


\begin{theorem}\label{struc}
 Let $A$ be a noncommutative half-commutative orthogonal Hopf algebra. Then there exists a self-transpose doubly non-real compact group $G$
 with $\mathbb T \subset G \subset U_n$ such that $A \simeq \mathcal A_*(G)$. 
\end{theorem}

\begin{proof} Let $A$ be a noncommutative  half-commutative orthogonal Hopf algebra.  
 The proof is divided into two steps.
 
 Step 1. In this preliminary step, we first write a convenient presentation for $A$. By Lemma \ref{exacthalf}
there exist surjective Hopf $*$-algebra maps 
$$A_o^*(n)  \overset{f} \rightarrow  A \overset{p}\rightarrow \mathbb C\mathbb Z_2$$ 
with $pf(v_{ij})=\delta_{ij}s$. We denote by $V$ the comodule over $A_o^*(n)$ corresponding to the matrix $v=(v_{ij})\in M_n(A_o^*(n))$, with its standard basis $e_1, \ldots, e_n$. To any linear map 
$\underline{\lambda} : \mathbb C \rightarrow V^{\otimes m}$, 
$$\underline{\lambda}(1) = \sum_{i_1, \ldots , i_m}\lambda(i_1, \ldots, i_m)e_{i_1} \otimes \cdots \otimes e_{i_m}$$
 we associate families $X(\underline{\lambda})$ and $X'(\underline{\lambda})$ of elements of $A_o^*(n)$ defined
by 
$$X(\underline{\lambda})= \{\sum_{j_1, \ldots ,j_m} v_{i_1j_1}\cdots  v_{i_mj_m}\lambda(j_1, \ldots, j_m) -\lambda(i_1, \ldots , i_m)1, \ i_1, \ldots ,i_m \in \{1, \ldots , n\}\}$$
$$X'(\underline{\lambda})= \{\sum_{j_1, \ldots ,j_m} v_{j_mi_m}\cdots  v_{j_1i_1}\lambda(j_1, \ldots, j_m) -\lambda(i_1, \ldots , i_m)1, \ i_1, \ldots ,i_m \in \{1, \ldots , n\}\}$$
These elements generate a $*$-ideal in $A_o^*(n)$, which is in fact a Hopf $*$-ideal, that we denote by $I_{\underline{\lambda}}$.
We also view $V$ as an $A$-comodule via $f$, and 
the map $\underline{\lambda}$ is a morphism of $A$-comodules if and only if $f(I_{\underline \lambda})=0$.
Now given a family $\mathcal C$ of linear maps $\mathbb C \rightarrow V^{\otimes m}$, $m \in \mathbb N$, we denote by $I_{\mathcal C}$ the Hopf $*$-ideal of $A_o^*(n)$ generated by all the elements of $X(\underline{\lambda})$ and $X'(\underline{\lambda})$, $\underline{\lambda} \in \mathcal C$. It follows from Woronowicz Tannaka-Krein duality \cite{wo2} that $f$ induces an isomorphism $A_o^*(n)/I_{\mathcal C} \simeq A$ for a suitable set 
$\mathcal C$ of morphisms of $A$-comodules (typically $\mathcal C$ is a family of morphisms that generate
the tensor category of corepresentations of $A$).

Step 2. We now construct a compact group $G$ with $\mathbb T \subset G \subset U_n$. We start with a presentation $A_o^*(n)/I_{\mathcal C} \simeq A$ as in Step 1. Note that the existence of the map $p : A \longrightarrow\mathbb C\mathbb Z_2$ implies that  for 
$\underline{\lambda} : \mathbb C \longrightarrow V^{\otimes m}$, if $\underline{\lambda} \not=0$ and $\underline{\lambda} \in \mathcal C$, then $m$ is even (evaluate $p$ on the elements of $X(\underline{\lambda})$).
We associate to $\underline{\lambda} : \mathbb C \longrightarrow V^{\otimes 2m} \in \mathcal C$
the following families of elements in $\mathcal R(U_n)$
\begin{align*}X_1(\underline{\lambda})= \{\sum_{j_1, \ldots ,j_{2m}} u_{i_1j_1}u_{i_2j_2}^*\cdots u_{i_{2m-1}j_{2m-1}} u_{i_{2m}j_{2m}}^*&\lambda(j_1, \ldots, j_{2m}) -\lambda(i_1, \ldots , i_{2m})1, 
\\ & i_1, \ldots ,i_{2m} \in \{1, \ldots , n\} \}
\end{align*}
\begin{align*}X_1'(\underline{\lambda})= \{\sum_{j_1, \ldots ,j_{2m}} u_{j_1i_1}^*u_{j_2i_2}\cdots u_{j_{2m-1}i_{2m-1}}^* u_{j_{2m}i_{2m}}&\lambda(j_1, \ldots, j_{2m}) -\lambda(i_1, \ldots , i_{2m})1, 
\\ & i_1, \ldots ,i_{2m} \in \{1, \ldots , n\} \}
\end{align*}
\begin{align*}X_2(\underline{\lambda})= \{\sum_{j_1, \ldots ,j_{2m}} u_{i_1j_1}^*u_{i_2j_2}\cdots u_{i_{2m-1}j_{2m-1}}^* u_{i_{2m}j_{2m}}&\lambda(j_1, \ldots, j_{2m}) -\lambda(i_1, \ldots , i_{2m})1, 
\\ & i_1, \ldots ,i_{2m} \in \{1, \ldots , n\} \}
\end{align*}
\begin{align*}X_2'(\underline{\lambda})= \{\sum_{j_1, \ldots ,j_{2m}} u_{j_1i_1}u_{j_2i_2}^*\cdots u_{j_{2m-1}i_{2m-1}} u_{j_{2m}i_{2m}}^*&\lambda(j_1, \ldots, j_{2m}) -\lambda(i_1, \ldots , i_{2m})1, 
\\ & i_1, \ldots ,i_{2m} \in \{1, \ldots , n\} \}
\end{align*}
Now denote by $J_{\mathcal C}$ the $*$-ideal of $\mathcal R(U_n)$ generated by the elements of  $X_1(\underline{\lambda})$, $X_1'(\underline{\lambda})$, $X_2(\underline{\lambda})$ and  $X'_2(\underline{\lambda})$ for all the elements $\underline{\lambda} \in \mathcal C$.
In fact $J_\mathcal C$ is a Hopf $*$-ideal and we define $G$ to be the compact group $G \subset U_n$ such that $\mathcal R(G) \simeq \mathcal R(U_n)/J_{\mathcal C}$. The existence of a Hopf $*$-algebra map 
$\rho : \mathcal R(G) \longrightarrow \mathbb C\mathbb Z$, $u_{ij} \longmapsto \delta_{ij}t$, where $t$ denotes a generator of $\mathbb Z$, is straightforward, and thus $\mathbb T \subset G$. Also it is easy to chek the existence of a Hopf $*$-algebra map $\mathcal R(G) \longrightarrow \mathcal R(G)$, $u_{ij} \longmapsto u_{ij}^*$, and this show that $G$ is self-transpose. 
We have, by Proposition \ref{rep}, a Hopf $*$-algebra map $\pi : A_o^*(n) \longrightarrow \mathcal A_*(G)$, $v_{ij} \longmapsto u_{ij}s$. It is a direct verification to check that $\pi$ vanishes on $I_\mathcal C$, so induces a Hopf $*$-algebra map $\overline{\pi} : A \longrightarrow \mathcal A_*(G)$.
We still denote by $v_{ij}$ the element $f(v_{ij})$ in $A$.
We get  a commutative diagram with pre-exact rows
$$\begin{CD}
\mathbb C@>>>P_vA@>{i}>>A@>{p}>>\mathbb C\mathbb Z_2 @>>>\mathbb C\\
@.@VV{\overline{\pi}_|}V@VV{\overline{\pi}}V@|@.\\
\mathbb C@>>>\mathcal R(G/\mathbb T)@>{j}>>\mathcal A_*(G)@>{q}>>\mathbb C\mathbb Z_2@>>>\mathbb C\\
\end{CD}$$
where the sequence on the top row is the one of Lemma \ref{exacthalf} and  the sequence on the lower row is the one of Proposition \ref{exact}. To prove that $\overline{\pi}$ is an isomorphism, we just have,
by the short five-lemma for cosemisimple Hopf algebra \cite{bbcc},   
to prove that $\overline{\pi}_| : P_vA \longrightarrow
\mathcal R(G/\mathbb T)$ is an isomorphism. 
Let $J'_{\mathcal C}$ be the $*$-ideal of $\mathcal R(PU_n)$ generated by the elements of  
$X_1(\underline{\lambda})$, $X_1'(\underline{\lambda})$, $X_2(\underline{\lambda})$ and  $X_2(\underline{\lambda})$ for all the elements $\underline{\lambda} \in \mathcal C$.
It is clear, using the $\mathbb Z$-grading on $\mathcal R(G)$ induced by the inclusion $\mathbb T \subset G$ and the fact that
$J_\mathcal C$ is generated by elements of degree zero, that $J'_{\mathcal C} = J_{\mathcal C} \cap \mathcal R(PU_n)$, so 
$\mathcal R(G/\mathbb T) \simeq \mathcal R(PU_n)/J'_{\mathcal C}$. But then
the natural $*$-algebra map $\mathcal R(PU_n) \longrightarrow P_vA$ (Lemma \ref{pun}) vanishes on $J'_{\mathcal C}$,
and hence induces a $*$-algebra map $\mathcal R(G/\mathbb T) \longrightarrow P_vA$, which is an inverse for
$\overline{\pi}_|$. Hence $\overline{\pi}$ is an isomorphism, and the algebra $A$ being noncommutative, it follows from Proposition \ref{rep} that $G$ is doubly non-real.
\end{proof}

Note that the proof of Theorem \ref{struc} also provides a method to find the compact group $G$ from the half-commutative orthogonal Hopf algebra $A$.

\begin{example}{\rm 
On can check, by following the proof of  Theorem \ref{struc}, that the hyperoctaedral Hopf algebra
$A_h^*(n)$ is isomorphic to $\mathcal A_*(K_n)$, where $K_n$ is the subgroup of $U_n$ formed by matrices having exactly one non-zero element on each column and line (with $K_n \simeq \mathbb T^n \rtimes S_n$).}
\end{example}

\begin{remark}{\rm 
 Let $H \subset G \subset U_n$ be self-transpose compact subgroups. The inclusion $H \subset G$ induces a surjective Hopf $*$-algebra map $\mathcal A_*(G) \rightarrow \mathcal A_*(H)$, compatible with the exact sequence in Proposition \ref{exact}. 
 Thus if the inclusion $H \subset G$ induces an isomorphism $H/H \cap \mathbb T \simeq G/G \cap \mathbb T$, the short five lemma ensures that $\mathcal A_*(G) \simeq \mathcal A_*(H)$.  
In particular $\mathcal A_*(U_n) \simeq \mathcal A_*(SU_n)$. }
\end{remark}

We now propose a tentative orthogonal half-liberation for the unitary group. 
In fact another possible half-liberation of $U_n$ has already been proposed in \cite{bhdada},
using the symbol $A_u^*(n)$.
We shall use the notation
$A_u^{**}(n)$ for the object we construct, which is different from the one in  \cite{bhdada}.

\begin{example}{\rm 
 Let $A_u^{**}(n)$ be the quotient of $A_u(n)$ by the ideal generated by the elements
 $$abc-cba, \ a,b,c, \in \{u_{ij}, u_{ij}^*\}$$
 Then  $A_u^{**}(n)$ is isomorphic with
 $\mathcal A_*(U_{2,n})$,  where $U_{2,n}$ is the subgroup of $U_{2n}$ consisting of unitary matrices of the form
 $$\begin{pmatrix}
    A & B \\
    -B & A
   \end{pmatrix}, \ A,B \in M_n(\mathbb C)$$
   and hence is  a half-commutative orthogonal Hopf algebra,
 }
\end{example}

\begin{proof}
Let $\omega \in \mathbb C$ be a primitive 4th root of unity.  We start with the probably well-known surjective Hopf $*$-algebra map
\begin{align*}
 A_o(2n) &\longrightarrow A_u(n) \\
 x_{i,j}, x_{n+i,n+j} &\longmapsto \frac{u_{ij}+u_{ij}^*}{2}, \ i,j \in \{1, \ldots , n\} \\
x_{n+i,j}  &\longmapsto \frac{u_{ij}-u_{ij}^*}{2\omega}, \ i,j \in \{1, \ldots , n\} \\
x_{i,n+j}  &\longmapsto \frac{u_{ij}^*-u_{ij}}{2\omega}, \ i,j \in \{1, \ldots , n\}
\end{align*}
where $x_{i,j}$ denote the standard generators of $A_o(2n)$. 
It is clear that it induces a surjective Hopf $*$-algebra map 
$A_o^*(2n) \longrightarrow A_u^{**}(n)$, and hence $A_u^{**}(n)$ is a half-commutative orthogonal Hopf algebra.

Let $J$ be the ideal of $A_o^*(2n)$ generated by the elements
$$v_{i,j}-v_{n+i,n+j}, \ v_{n+i,j}+v_{i,n+j}, \ i,j \in \{1, \ldots , n\}$$
(where $v_{i,j}$ denotes the class of $x_{ij}$ in $A_o^*(n)$).
Then $J$ is a Hopf $*$-ideal in $A_o^*(2n)$ and the previous Hopf $*$-algebra map induces an isomorphism
$A_o^*(2n)/J \simeq A_u^{**}(n)$ (the inverse sends $u_{ij}$ to $x_{ij} + \omega x_{n+i,j}$). 
Now having the presentation $A_o^*(2n)/J \simeq A_u^{**}(n)$, the proof of Theorem \ref{struc} 
yields $ A_u^{**}(n) \simeq \mathcal A_*(U_{2,n})$.
\end{proof}

\section{Representation theory}

In this section we describe the fusion rules of $\mathcal A_*(G)$ for any compact group $G$
(as usual by fusion rules we mean the set of isomorphism classes of simple comodules together with the decomposition of tensor products of simple comodules into simple constituents).
Thanks to Theorem \ref{struc}, this gives a description of the fusion rules of any half-commutative orthogonal Hopf algebra.

If $A$ is a cosemisimple Hopf algebra, we denote by ${\rm Irr}(A)$ the set of simple (irreducible) comodules over $A$. If $A= \mathcal R(G)$ for some compact group, then ${\rm Irr}(\mathcal R(G))={\rm Irr}(G)$, the set of isomorphism classes of irreducible representations of $G$.
By a slight abuse of notation, for a simple $A$-comodule $V$, we write $V \in {\rm Irr}(A)$.

Let $G \subset U_n$ be a self-transpose compact subgroup. Recall that the transposition 
induces an involutive compact group automorphism
 \begin{align*}
 \sigma : G &\longrightarrow G \\
  g &\longmapsto (g^t)^{-1} = \overline{g}
 \end{align*}
For $V \in {\rm Irr}(G)$, we denote by $V^\sigma$ the (irreducible) representation of $G$ induced 
by the composition with $\sigma$. If $U$ is the fundamental $n$-dimensional representation of $G$, then
$U^\sigma \simeq \overline{U}$.

We begin by recalling the description of the fusion rules for the crossed product $\mathcal R(G) \rtimes \mathbb C\mathbb Z_2$. This is certainly well-known (see e.g. \cite{wa2}, Theorem 3.7)). 

\begin{proposition}\label{fuscrossed}
 Let $G \subset U_n$ be a self-transpose compact subgroup. Then there is a bijection
 $${\rm Irr}(\mathcal R(G) \rtimes \mathbb C\mathbb Z_2) \simeq {\rm Irr}(G) \amalg{\rm Irr}(G)$$
 More precisely, if $X \in {\rm Irr}(\mathcal R(G) \rtimes \mathbb C\mathbb Z_2)$, then there exists 
 a unique $V \in {\rm Irr}(G)$ with either $X \simeq V$ or $X \simeq V \otimes s$. For $V,W \in {\rm Irr}(G)$, we have
 $$V \otimes (W \otimes s) \simeq (V \otimes W) \otimes s, \
  (V \otimes s) \otimes W\simeq (V \otimes W^\sigma) \otimes s , \
  (V \otimes s) \otimes (W\otimes s) \simeq V \otimes W^\sigma $$
\end{proposition}

\begin{proof}
 The description of the simple comodules follows in a straightforward manner from the fact that $\mathcal R(G) \rtimes \mathbb C\mathbb Z_2 = \mathcal R(G) \otimes \mathbb C\mathbb Z_2$ as coalgebras.
 The tensor product decompositions are obtained by using character theory, see \cite{wo1} or \cite{ks}.
\end{proof}

\begin{remark}{\rm 
 If $G \subset U_n$ is connected and has a maximal torus $T$ of $G$ contained in $\mathbb T^n$,
 it follows from highest weight theory that $V^\sigma \simeq \overline{V}$ for any $V \in {\rm Irr}(G)$.
 We do not know if this is still true without these assumptions. }
\end{remark}

To express the fusion rules of $\mathcal A_*(G)$, we need more notation.
 Let $G \subset U_n$ be a compact subgroup, and denote by $U$ the fundamental $n$-dimensional representation of $G$. For $m \in \mathbb Z$, we put
 $${\rm Irr}(G)_{[m]} = \{V \in {\Irr}(G), \ V \subset U^{\otimes m} \otimes (U \otimes \overline{U})^{\otimes l}, \ {\rm for} \ {\rm some} \ l \in \mathbb N\}$$  
where $U^{ \otimes 0} = \mathbb C$ and for $m<0$ $U^{\otimes m} = \overline{U}^{\otimes -m}$.  

Now if $V \in {\rm Irr}(G)_{[0]}$, then $V \in {\rm Irr}(G/G \cap \mathbb T)$ (see Lemma \ref{pun}), and since $\mathcal R(G/G \cap \mathbb T) \subset \mathcal A_*(G)$, we get an element
in ${\rm Irr}(\mathcal A_*(G))$, still denoted $V$.

If $V \in {\rm Irr}(G)_{[1]}$, then $V \subset U \otimes (U \otimes \overline{U})^{\otimes l}$, for  some  $l \in \mathbb N$, and hence the coefficients of $V \otimes s$ belong to $\mathcal A_*(G)$.
Thus we get an element of  ${\rm Irr}(\mathcal A_*(G))$, denoted $Vs$. 

\begin{corollary}
 Let $G \subset U_n$ be a self-transpose compact subgroup. Then the map
 $$ {\rm Irr}(G)_{[0]} \amalg{\rm Irr}(G)_{[1]}  \longrightarrow{\rm Irr}(\mathcal A_*(G))$$
$$
  V \longmapsto \begin{cases}
                   V & \text{if}  \ V \in {\rm Irr}(G)_{[0]} \\
		   Vs  &\text{if} \ V \in {\rm Irr}(G)_{[1]}
                  \end{cases} $$
 is a bijection. Moreover, for $V \in {\rm Irr}(G)_{[0]}$, $W,W' \in {\rm Irr}(G)_{[1]}$, we have
 $$V  \otimes Ws \simeq (V \otimes W)s, \
 Ws \otimes V \simeq  (W \otimes V^\sigma)s, \
 Ws \otimes W's \simeq W \otimes W'^{\sigma}, \
 \overline{Ws} \simeq \overline{W}^\sigma s$$
\end{corollary}

\begin{proof}
The existence of the map follows from the discussion before the corollary, while injectivity comes
from Proposition \ref{fuscrossed}.
Note that for $V \in  {\rm Irr}(G)_{[m]}$, $V' \in {\rm Irr}(G)_{[m']}$, the simple constituents of 
$V \otimes V'$ all belong to ${\rm Irr}(G)_{[m+m']}$, and that $V^\sigma \in {\rm Irr}(G)_{[-m]}$.
So the isomorphisms in the statement (that all come from the isomorphisms of Proposition \ref{fuscrossed})
yield decompositions into simple $\mathcal A_*(G)$-comodules. Thus we have a family of simple $\mathcal A_*(G)$-comodules,
stable under decompositions of tensor products and conjugation, and that contains the fundamental
comodule $Us$: we conclude (e.g. from the orthogonality relations \cite{wo1}, \cite{ks})
that we have all the simple comodules.
\end{proof}





\end{document}